\documentclass[11pt]{amsart}
\usepackage{graphicx}
\usepackage{amsmath,amsthm}
\usepackage{amssymb}
\usepackage{amscd}
\usepackage{verbatim}
\usepackage{color}
\usepackage[colorlinks]{hyperref}
\usepackage{tikz}

\usepackage[bitstream-charter]{mathdesign}
\usepackage{eucal}

\usepackage{color}
\usepackage{t1enc}

\newcommand{\mbb}{\mathbb}
\newcommand{\mf}{\mathfrak}

\newcommand{\mc}{\mathcal}

\newcommand{\eps}{\varepsilon}
\newcommand{\de}{\delta}
\newcommand{\la}{\langle}
\newcommand{\ra}{\rangle}
\newcommand{\al}{\alpha}

\newcommand{\lam}{\lambda}
\newcommand{\om}{\omega}

\newcommand{\0}{\emptyset}

\newcommand{\bc}{\begin{center}}
\newcommand{\ec}{\end{center}}

\newcommand{\rest}{\!\upharpoonright\!}

\newcommand{\sub}{\subseteq}

\newtheorem*{claim}{Claim}
\newtheorem{thm}{Theorem}[section]
\newtheorem{lem}[thm]{Lemma}
\newtheorem{prop}[thm]{Proposition}
\newtheorem{cor}[thm]{Corollary}
\newtheorem{fact}[thm]{Fact}

\theoremstyle{definition}
\newtheorem{que}[thm]{Question}
\newtheorem{df}[thm]{Definition}
\newtheorem{exa}[thm]{Example}

\newtheorem{rem}[thm]{Remark}

\title{Representations of ideals in Polish groups and in Banach spaces}

\author{Piotr Borodulin--Nadzieja}
\address[Piotr Borodulin-Nadzieja]{Instytut Matematyczny, Uniwersytet Wroc\l awski, pl. Grunwaldzki 2/4, 50-384 Wroc\l aw, Poland}
\email{pborod@math.uni.wroc.pl}

\author{Barnab\'as Farkas}
\address[Barnab\'as Farkas]{Instytut Matematyczny, Uniwersytet Wroc\l awski, pl. Grunwaldzki 2/4, 50-384 Wroc\l aw, Poland; later Kurt G\"{o}del Research Center, Vienna, Austria}
\email{barnabasfarkas@gmail.com}

\author{Grzegorz Plebanek}
\address[Grzegorz Plebanek]{Instytut Matematyczny, Uniwersytet Wroc\l awski, pl. Grunwaldzki 2/4, 50-384 Wroc\l aw, Poland}
\email{grzes@math.uni.wroc.pl}


\def\Ubf#1{{\baselineskip=0pt\vtop{\hbox{$#1$}\hbox{$\sim$}}}{}}

\begin{document}

\begin{abstract}
We investigate ideals of the form  $\{A\subseteq\om\colon \sum_{n\in A}x_n$ is unconditionally convergent$\}$ where $(x_n)_{n\in\om}$ is a sequence in a Polish group or in a Banach space. If an ideal on $\omega$ can be seen in this form for some
sequence in $X$, then we say that it is representable in $X$.


After numerous examples we show the following theorems: (1) An ideal is representable in a Polish Abelian group iff it is an analytic P-ideal. (2) An ideal is representable in a Banach space iff it is a non-pathological analytic P-ideal.

We focus on the family of ideals representable in $c_0$. We characterize this property via the defining sequence of measures. We prove that the trace of the null ideal, Farah's ideal, and Tsirelson ideals are not representable in $c_0$, and that a
tall $F_\sigma$ P-ideal is representable in $c_0$ iff it is a summable ideal. Also, we provide an example of a peculiar ideal which is representable in $\ell_1$ but not in $\mathbb{R}$.

Finally, we summarize some open problems of this topic.
\end{abstract}

\maketitle

\section{Introduction}

Recall that an ideal $\mathcal{I}$ on $\om$ is {\em summable} if it is defined by a measure, i.e. there is a (mass) function $m\colon \om\to [0,\infty)$ with $\sum_{i\in\om}m(i)=\infty$ such that
\[ I \in \mathcal{I} \iff \sum_{i\in I} m(i) < \infty. \]
In this case, we write $\mc{I}=\mc{I}_m$. Summable ideals, together with density ideals, are the flagship examples of {\em analytic P-ideals} on $\omega$. However, this class contains also ideals which are, from the
combinatorial viewpoint, very far from summable and density ideals (see examples in the next section).

In this article, we consider some natural generalizations of summable ideals. Consider a space $X$ equipped in enough structure to
speak about convergence of series, e.g. a  Polish Abelian group or a Banach space. We say that an ideal $\mc{J}$ on $\omega$ is \emph{representable in} $X$ if
there is a function $m\colon \om \to X$ such that
\[ I \in \mathcal{J} \iff \sum_{i\in I} m(i) \mbox{ converges unconditionally in $X$. } \]
If $X$ is complete, then the family of sets defined by the right part of the formula above is an ideal on $\om$.

Particular examples of non-summable ideals of this form appeared already in \cite{farah-tsirelson} and \cite{veli-tsirelson} (so called {\em Tsirelson} ideals).  We deal rather with the \textbf{general questions}: Which ideals can be represented in (certain)  Polish Abelian groups and in (certain) Banach spaces?

In Section \ref{prelim} we present a short survey on the basics of analytic P-ideals and on the related main tools we will need.

In Section \ref{representation} we introduce the notion of representations of ideals, and present some examples of representations of classical analytic P-ideals.

In Section
\ref{main} we prove that an ideal is representable in some Polish Abelian group iff it is an analytic P-ideal; and an ideal  is representable in a Banach space iff it is, additionally, {\em non-pathological}.

Recall the theorem due to Solecki which says that each analytic P-ideal can be defined by using a lower semicontinuous submeasure.
Morally, this result says that each analytic P-ideal is in a sense similar to the density ideals.
Indeed, many facts about the density ideals can be generalized almost automatically by considering arbitrary submeasures instead of those given by the density functions.

Partially, our research has a similar motivation. We investigate how much analytic P-ideals resemble the summable ideal. Although our results can be interpreted as an indication that ``in a sense'' each analytic P-ideal (especially non-pathological) is
summable, one should not expect here as strong consequences as in the case of Solecki theorem.
One of the main reason is that there is no general theory of summable ideals.
However, we believe  that
\begin{itemize}
\item[(i)] our approach reveals some ``geometric'' properties of non-pathological ideals and therefore it can be helpful in their classification;
\item[(ii)] these methods can be useful in providing new interesting examples of non-pathological analytic P-ideals;
\item[(iii)] representability of certain ideals in Banach spaces can be seen as a combinatorial property of the space itself and this may lead us to develop new methods in the theory of Banach spaces.
\end{itemize}

A few more words on (iii). For example, one can ask which ideals are represented in concrete Banach spaces. It seems that the characterization of representability in $c_0$ is one of the most interesting questions here. We have not been able to characterize fully the ideals representable in $c_0$
but in Section \ref{c0} we prove
that tall $F_\sigma$ ideals are not representable in $c_0$ (if we exclude the ``trivial'' case of summable ideals). We also show that the trace of the null ideal is not representable in $c_0$. These results suggest that ideals representable in $c_0$ are more connected to density (like) ideals.

In contrast, the ideals represented in $\ell_1$ should be more close to summable ideals. Actually, ideals representable in $\ell_1^+$ are exactly the summable ideals. In Section \ref{ell_1} we show that this is no longer true for $\ell_1$: we
present an $F_\sigma$ ideal which is representable in $\ell_1$ but which is not summable.


In Section \ref{questions} we list some of our related open questions with additional explanations.


\section{Acknowledgements}
The second author was supported by
Hungarian National Foundation for Scientific Research grant nos.
83726 and 77476. Furthermore, during the work on this paper he held a post-doctoral research position first at the Mathematical Institute of University of Wroc\l aw; and then at the Kurt G\"{o}del Research Center in Vienna where he was supported by the Austrian Science Fund (FWF) grant no. P25671.

\section{Preliminaries}\label{prelim}

Denote by $\mathrm{Fin}$ the ideal of finite subsets of $\om$, $\mathrm{Fin}=[\om]^{<\om}$. If $\mc{I}$ is
an ideal on $\om$, then we always assume that it is {\em proper}, i.e. $\om\notin
\mc{I}$, and $\mathrm{Fin}\subseteq\mc{I}$. Write $\mc{I}^+=\mc{P}(\om)\setminus \mc{I}$ for the family of {\em $\mc{I}$-positive} sets and
$\mc{I}^*=\{\om\setminus X\colon X\in\mc{I}\}$ for the {\em dual filter} of $\mc{I}$. If $X\in\mc{I}^+$ then the {\em restriction} of $\mc{I}$ to $X$ is $\mc{I}\upharpoonright X=\{A\in\mc{I}\colon A\subseteq X\}$. An ideal $\mc{I}$ on $\om$ is {\em $F_\sigma$, Borel, analytic} if
$\mc{I}\subseteq\mc{P}(\om)\simeq 2^\om$ is an $F_\sigma$, Borel, analytic set in the
usual product topology of the Cantor-set.

$\mc{I}$ is a {\em
P-ideal} if for each countable $\mc{C}\subseteq\mc{I}$ there is an
$A\in\mc{I}$ such that $C\subseteq^* A$ for each $C\in\mc{C}$ (where $C\subseteq^* A$ iff $C\setminus A$ is finite).
In other words, $\mc{I}$ is a P-ideal iff the preordered set $(\mc{I},\subseteq^*)$ is $\sigma$-directed.

$\mc{I}$ is {\em tall} (or {\em dense}) if each infinite subset of $\om$ contains
an infinite element of $\mc{I}$.

If $\mc{A}\subseteq\mc{P}(\om)$ then the ideal {\em generated} by $\mc{A}$ 
is
\[ \mathrm{id}(\mc{A})=\Big\{X\subseteq\om:\exists\;\mc{A}'\in [\mc{A}]^{<\om}\;X\subseteq^*\bigcup\mc{A}'\Big\}.\]

In our investigations Borel P-ideals play the most important role. We show some classical examples of these ideals (for more see \cite{Hrusak} or \cite{Meza}):

\smallskip
The mentioned above {\em summable ideal}:
\[ \mc{I}_{1/n}=\left\{A\subseteq\om:\sum_{n\in A}
\frac{1}{n+1}<\infty\right\}.\]
$\mc{I}_{1/n}$ is a tall $F_\sigma$ P-ideal. In general, if $h\colon \om\to [0,\infty)$ is such that $\sum_{n\in\om}h(n)=\infty$ then the {\em summable ideal associated to $h$} is $\mc{I}_h=\big\{A\subseteq\om\colon \sum_{n\in A}h(n)<\infty\big\}$. It is also an $F_\sigma$ P-ideal which is tall iff $h(n)\to 0$.

\smallskip
The {\em density zero ideal}:
\[ \mc{Z}=\left\{A\subseteq\om\colon \mathop{\lim}\limits_{n\to\infty}\frac{|A\cap
n|}{n}=0\right\}=\left\{A\subseteq\om\colon \lim_{n\to\infty}\frac{|A\cap [2^n,2^{n+1})|}{2^n}=0\right\}.\]
$\mc{Z}$ is a tall $F_{\sigma\delta}$ $P$-ideal. In general, if $\vec{\mu}=(\mu_n)_{n\in\om}$ is a sequence of measures on $\om$ with pairwise disjoint finite supports and $\limsup_{n\to\infty}\mu_n(\om)>0$, then the {\em density ideal associated to
$\vec{\mu}$} is $\mc{Z}_{\vec\mu}=\big\{A\subseteq\om:\mu_n(A)\to 0\big\}$. It is always an $F_{\sigma\delta}$ P-ideal which is tall iff $\max\{\mu_n(\{k\})\colon k\in\om\}\xrightarrow{n\to\infty} 0$.

\smallskip
{\em Generalized density ideals}: If $\vec{\varphi}=(\varphi_n)_{n\in\om}$ is a sequence of \textbf{sub}measures on $\om$ (see below for the definition of a submeasure) with pairwise disjoint finite supports and $\limsup_{n\to\infty}\varphi_n(\om)>0$, then the {\em generalized density
ideal associated to $\vec{\varphi}$} is $\mc{Z}_{\vec\varphi}=\big\{A\subseteq\om\colon \varphi_n(A)\to 0\big\}$. $\mc{Z}_{\vec{\varphi}}$ is an $F_{\sigma\delta}$ P-ideal which is tall iff $\max\{\varphi_n(\{k\})\colon k\in\om\}\xrightarrow{n\to\infty} 0$.
\smallskip


The Fubini-product of $\{\emptyset\}$ and $\mathrm{Fin}$:
\[ \{\emptyset\}\otimes\mathrm{Fin}=\big\{A\subseteq\om\times\om:\forall\;n\in\om\;|(A)_n|<\om\big\},\]
where $(A)_n=\{m\colon (n,m)\in A\}$. It is a non-tall $F_{\sigma\delta}$ P-ideal. Observe that $\{\emptyset\}\otimes\mathrm{Fin}$ is a density ideal: for $n,m\in\om$ let $\mathrm{supp}(\mu_{n,m})=\{(n,m)\}$ and let $\mu_{n,m}(\{(n,m)\})=\frac{1}{n+1}$. It is easy to see that if $\vec{\mu}=(\mu_{n,m})_{n,m\in\om}$ then $\{\emptyset\}\otimes\mathrm{Fin}=\mc{Z}_{\vec{\mu}}$.

\smallskip
{\em Farah's ideal}: The following ideal is the simplest known example of an $F_\sigma$ P-ideal which is not a summable ideal (see \cite[Section 1.11]{farah}):
\[ \mc{J}_F=\left\{A\subseteq\om\colon \sum_{n\in\om}\frac{\min\{n,|A\cap [2^n,2^{n+1})|\}}{n^2}<\infty\right\}.\]

\smallskip
The {\em  trace of the null ideal} : Let $\mc{N}$ be the $\sigma$-ideal of subsets of $2^\om$ with measure zero (with respect to the usual product measure). The {\em $G_\delta$-closure} of a set $A\subseteq 2^{<\om}$ is $[A]=\big\{x\in
2^\om\colon \exists^\infty$ $n$ $x\rest n\in A\big\}$, a $G_\delta$ subset of $2^\om$. The trace of $\mc{N}$ is defined by
\[ \mathrm{tr}(\mc{N})=\big\{A\subseteq 2^{<\om}\colon [A]\in \mc{N}\big\}.\]
It is a tall $F_{\sigma\delta}$ P-ideal.

\begin{rem}
Observe that in some sense $\mc{I}_{1/n}\subseteq\mathrm{tr}(\mc{N})\subseteq\mc{Z}$: let $\mc{I}_\mathrm{tree}$ be the ``tree version'' of the summable ideal, that is,
\[ \mc{I}_\mathrm{tree}=\bigg\{A\subseteq 2^{<\om}\colon \sum_{s\in A} 2^{-|s|}<\infty\bigg\}.\]
Then clearly $\mc{I}_{1/n}$ and $\mc{I}_\mathrm{tree}$ are isomorphic (by the most natural enumeration of $2^{<\om}$), and  $\mc{I}_\mathrm{tree}\subseteq\mathrm{tr}(\mc{N})$. Furthermore, if $\mc{Z}_\mathrm{tree}$ is the tree version of the density zero ideal,
\[ \mc{Z}_\mathrm{tree}=\Big\{A\subseteq 2^{<\om}\colon \lim_{n\to\infty}\frac{|A\cap 2^n|}{2^n}=0\Big\},\]
then it is isomorphic to $\mc{Z}$ and $\mathrm{tr}(\mc{N})\subseteq\mc{Z}_\mathrm{tree}$.
\end{rem}


\smallskip
We will apply Solecki's representation of analytic P-ideals.
A function $\varphi\colon \mc{P}(\om)\to [0,\infty]$
is a {\em submeasure} (on $\om$) if $\varphi(\emptyset)=0$; if $X,  Y\subseteq\om$ then ${\varphi}(X)\le \varphi(X\cup Y)\le \varphi(X)+\varphi(Y)$; and $\varphi(\{n\})<\infty$ for $n\in\om$. $\varphi$ is {\em lower semicontinuous} (lsc, in short) if $\varphi(X)=\lim_{n\to\infty}\varphi(X\cap n)$ for each $X\subseteq\om$.

If $\varphi$ is an lsc submeasure then for $X\subseteq\om$ let $\|X\|_\varphi=\lim_{n\to\infty}\varphi(X\setminus n)$; and let
\begin{align*}
\mathrm{Exh}(\varphi) = &  \big\{X\subseteq \om:\|X\|_\varphi=0\big\},\\
\mathrm{Fin}(\varphi) = & \big\{X\subseteq \om:\varphi(X)<\infty\big\}.
\end{align*}
It is easy to see that if $\mathrm{Exh}(\varphi)\ne\mc{P}(\om)$, then it is an $F_{\sigma\delta}$ P-ideal, which is tall iff $\varphi(\{n\})\to 0$. Similarly, if $\mathrm{Fin}(\varphi)\ne\mc{P}(\om)$ then it is an $F_\sigma$ ideal. Clearly, $\mc{I}_{\varphi(\{\cdot\})}\subseteq\mathrm{Exh}(\varphi)\subseteq\mathrm{Fin}(\varphi)$ always holds, where $\mc{I}_{\varphi(\{\cdot\})}$ stands for the summable ideal generated by the sequence $\varphi(\{n\})$.

The next theorem provides one of the most important tools in the theory of analytic P-ideals.

\begin{thm} {\em (see \cite[Thm. 3.1.]{So})}\label{sochar}
Let $\mc{J}$ be an ideal on $\om$. Then the following are equivalent:
\begin{itemize}
\item[(i)] $\mc{J}$ is an analytic $P$-ideal;
\item[(ii)] $\mc{J}=\mathrm{Exh}(\varphi)$ for some (finite) lsc submeasure $\varphi$;
\item[(iii)] $\mc{J}$ is Polishable, that is, there is a Polish group topology on $\mc{J}$ with respect to the usual group operation such that the Borel structure of this topology coincides with the Borel structure inherited from $\mc{P}(\om)$.
\end{itemize}
Furthermore, $\mc{J}$ is an $F_\sigma$ P-ideal iff $\mc{J}=\mathrm{Exh}(\varphi)=\mathrm{Fin}(\varphi)$ for some lsc submeasure $\varphi$.
\end{thm}

In particular, analytic P-ideals are $F_{\sigma\delta}$.

The implication (ii)$\Rightarrow$(iii) is not difficult: for $A,B\in\mathrm{Exh}(\varphi)$ let $d_\varphi(A,B)=\varphi(A\triangle B)$. Then it is easy to see that $d_\varphi$ is a (translation) invariant complete metric (we can assume that $\varphi(\{n\})>0$  for every $n$), the generated topology is finer than the subspace topology, and  $\mathrm{Borel}(\mathrm{Exh}(\varphi),d_\varphi)=\mathrm{Borel}(\mc{P}(\om))\upharpoonright \mathrm{Exh}(\varphi)$.

If we refer to $\mathrm{Exh}(\varphi)$ as complete metric group, then we mean that it is equipped with $d_\varphi$.

\begin{rem}\label{pleb}
Notice that the Polish topology on $\mathrm{Exh}(\varphi)$ does not depend on the choice of $\varphi$.
It follows from the fact that $\mathrm{Exh}(\varphi)=\mathrm{Exh}(\psi)$ if, and only if for every sequence $(A_n)_{n\in\om}$ of pairwise disjoint finite sets $(\varphi(A_n)\to 0\iff\psi(A_n)\to 0)$.
\end{rem}

The summable and (generalized) density ideals can be written of the form $\mathrm{Exh}(\varphi)$ very easily.

\smallskip
The definition of Farah's ideal explicitly contains the definition of a submeasure $\varphi$, and clearly $\mc{J}_F=\mathrm{Exh}(\varphi)=\mathrm{Fin}(\varphi)$.

\smallskip
We also show the standard presentation of $\mathrm{tr}(\mc{N})$ of the form $\mathrm{Exh}(\varphi)$. For every non-empty $A\subseteq 2^{<\om}$  let
\[ \varphi(A)=\sup\bigg\{\sum_{s\in B}2^{-|s|}:B\subseteq A\;\;\text{is an antichain}\bigg\}.\]
Notice that actually this supremum is maximum: for each $A\subseteq 2^{<\om}$ let $B_A$ be the antichain of the $\subseteq$-minimal elements of $A$, then $\varphi(A)=\sum_{s\in B_A}2^{-|s|}$. Then $\mathrm{tr}(\mc{N})=\mathrm{Exh}(\varphi)$.

\section{Generalization of summable ideals} \label{representation}

Let $G$ be a nontrivial Hausdorff topological Abelian  group (with the additive notation). We will use
the following basic notions from the theory of topological Abelian  groups:
\begin{itemize}
\item A {\em net} $(a_i)_{i\in I}$ in $G$ is a sequence in $G$ indexed by the underlying set of a directed poset $(I,\leq)$.
\item The net $(a_i)_{i\in I}$ {\em converges to $b\in G$} if for every neighborhood $U$ of $b$ there is an $i_0\in I$ such that $a_i\in U$ for every $i\geq i_0$. Clearly, a net has at most one limit.
\item A net $(a_i)_{i\in I}$ is a {\em Cauchy-net} if for every neighborhood $V$ of $0\in G$, there is a $j_0\in I$ such that $a_{j}-a_{j_0}\in V$ for every $j\geq j_0$ (this is a simplified but equivalent definition of Cauchy nets).
\item $G$ is {\em complete} if every Cauchy-nets converge (the reverse implication always holds).
\end{itemize}

Recall (see \cite{klee}) that if $G$ is metrizable
and complete,
then there is a compatible invariant (and hence complete) metric on $G$.

Using nets, we can define the {\em unconditional convergence} of infinite series in $G$: let $h\colon \om\to G$ be a sequence in $G$. Then we write $\sum_{n\in \om}h(n)=a\in G$ if
the net
\[ \sum h=\bigg(s_h(F)=\sum_{n\in F}h(n)\colon F\in [\om]^{<\om}\bigg)\;\;\;\text{ordered by}\;\subseteq\;\text{on}\;\;[\om]^{<\om}\] converges to $a$; in other words:
\bc $\forall$ open $U\ni a$ $\exists$ $F\in [\om]^{<\om}$ $\forall$ $E\in [\om]^{<\om}$ $\big(F\subseteq E\Rightarrow s_h(E)\in U\big)$.\ec
It is easy to see that $\sum_{n\in\om}h(n)=a$ iff $h(\pi(0))+h(\pi(1))+\dots+h(\pi(n))\xrightarrow{n\to\infty} a$ for every permutation $\pi$ of $\om$.

Similarly, the series associated to $h$ is {\em unconditionally  Cauchy} if the net $\sum h$ is Cauchy, i.e.  \bc $\forall$ open $V\ni 0$ $\exists$ $F\in [\om]^{<\om}$ $\forall$ $E\in [\om\setminus F]^{<\om}$ $s_h(E)\in V$.\ec

Now we are finally ready to introduce the main definition of the article. Assume that $\sum_{n\in \om}h(n)$ does not exist. Then the {\em (generalized) summable ideal associated to $h$}, $\mc{I}^G_h$ is the ideal generated by
\begin{align*} \mc{S}^G_h & =\bigg\{A\subseteq \om:\sum_{n\in A}h(n)\;\;\text{exists in}\;\;G\bigg\}\\ & =\bigg\{A\subseteq\om:A\;\;\text{is finite}\;\,\text{or}\;\,\sum h\!\upharpoonright\!A\;\,\text{is convergent in}\;\,G\bigg\}.
\end{align*}

Of course, $\mc{S}^G_h$ is not necessarily an ideal. It is always closed for taking unions but not necessary for taking subsets, see e.g. $G=\mathbb{Q}$ (with the usual addition) and let $h\colon \om\to\mbb{Q}$, $h(n)=\frac{1}{n+1}$. However, it is
easy to see the following.
\begin{fact}
If $G$ is complete, then $\mc{I}^G_h=\mc{S}^G_h$. If $G$ is complete and metrizable, then $\mc{I}^G_h$ is tall iff $h(n)\to 0\in G$.
\end{fact}

\begin{df}
We say that an ideal $\mc{J}$ on $\om$ is {\em representable in $G$} if there is an $h\colon \om\to G$ such that $\mc{J}=\mc{I}^G_h$.
If $\mathbf{C}$ is a class of topological  Abelian  groups then $\mc{J}$ is {\em $\mathbf{C}$-representable} if it is representable in a $G\in\mathbf{C}$.
\end{df}

For example, we can talk about Polish- or Banach-representable ideals.
Notice that we can always assume that $G$ is separable because essentially we are working in $\overline{\la \mathrm{ran}(h)\ra}$ (where $\la H\ra$ denotes the subgroup of $G$ generated by $H$), this clearly holds true for Banach spaces too because $\overline{\mathrm{span}}(\mathrm{ran}(h))$ is separable.

Let us see some examples:

\begin{exa} Summable ideals are exactly those ideals which are representable in $\mbb{R}$. For any $h\colon \om\to\mbb{R}$ the sum $\sum_{n\in A}h(n)$ exists (in the unconditional sense) iff $\sum_{n\in A}|h(n)|<\infty$, and hence $\mc{I}^\mbb{R}_h=\mc{I}^\mbb{R}_{|h|}=\mc{I}_{|h|}$.
Similarly, $\mc{J}$ is summable iff it is representable in $\mbb{R}^n$.
\end{exa}

We will frequently use the classical (real) sequence spaces $\ell_1$, $\ell_\infty$, and $c_0$. We assume that the reader is familiar with their basic properties.

\begin{exa}\label{Zinc0}
The ideal $\mc{Z}$ is representable in $c_0$. Let $h(0)=0$ and $h(n)=2^{-k}e_k$ iff $n\in [2^k,2^{k+1})$ where $e_k=(\delta_{k,m})_{m\in\om}$. Then $\mc{Z}=\mc{I}^{c_0}_h$.
\end{exa}

\begin{exa}
If $(G_n)_{n\in\om}$ is a sequence of non-trivial discrete Abelian  groups, then $\mc{J}$ is representable in $\prod_{n\in\om}G_n$ iff there is a countable (not necessarily infinite) family $\{X_n\colon n\in\om\}\subseteq [\om]^\om$ such that
\[ \mc{J}=\big\{A\subseteq\om\colon \forall\;n\in\om\;|A\cap X_n|<\om\big\}.\]
For example, $\{\emptyset\}\otimes\mathrm{Fin}$ has this property.
\end{exa}

\begin{exa}\label{tsirelson}
	Tsirelson ideals (see \cite{farah-tsirelson} and \cite{veli-tsirelson}) $\mathcal{T}$ have the following form
	\[ A \in \mathcal{T} \mbox{ iff } \sum_{n\in A} \alpha_n e_n \mbox{ is unconditionally convergent in } T, \]
	where $(e_n)_{n\in\om}$ is the standard basic sequence in $\ell_1$, $(\alpha_n)_{n\in\om} \in c_0 \setminus \ell_1$  is fixed, and $T$ is a Tsirelson space. Note that here $T$ can be understood either as the original Tsirelson space or as its dual. Of course every Tsirelson ideal is representable in a Tsirelson space.
	\end{exa}

\begin{exa}
Let $\mbb{T}$ be the group $\mbb{R}/\mbb{Z}$.
Notice that an ideal is representable in $\mbb{T}$ (or in $\mbb{T}^n$) iff it is a summable ideal. Indeed, if $\mc{I}_h$ is a summable ideal where $h\colon \om\to[0,\infty)$, then we can assume that $h\leq 1/2$ because $\mc{I}_h=\mc{I}_{h'}$ where $h'(n)=\min\{h(n),1/2\}$. It is easy to see that considering $h$ as a sequence in $\mbb{T}$, we obtain the same ideal.

If $h\colon \om\to\mbb{T}=[0,1)$, then it is not difficult to show that $\mc{I}^\mbb{T}_h=\mc{I}_g$ where $g(n)=h(n)$ if $h(n)\leq 1/2$ and $g(n)=1-h(n)$ else.
\end{exa}

\begin{exa}\label{exaXk} An ideal is representable in $\mbb{R}^\om$ iff it is an intersection of countable many summable ideals. Indeed, assume that $h\colon \om\to\mbb{R}^\om$, $h(n)=(x^n_k)_{k\in\om}$, and define $h_k\colon \om\to\mbb{R}$, $h_k(n)=x^n_k$ for $k,n\in\om$. Then $\mc{I}^{\mbb{R}^\om}_h=\bigcap_{k\in\om}\mc{I}_{h_k}$, and of course, the same idea works in the reverse direction too.

There are non $F_\sigma$ (and hence non summable) ideals which are representable in $\mbb{R}^\om$. Let $\{X_k\colon k\in\om\}$ be a partition of $\om$ into infinite sets, such that $\sum_{n\in X_k}\frac{1}{n+1}=\infty$, and let
\[ \mc{J}_0=\bigg\{A\subseteq\om\colon \forall\;k\;\sum_{n\in A\cap X_k}\frac{1}{n+1}<\infty\bigg\}.\]
Then $\mc{J}_0=\bigcap_{k\in\om}\mc{I}_{h_k}$ where $h_k(n)=\frac{\chi_{X_k}(n)}{n+1}$, and hence it is representable in $\mbb{R}^\om$. $\mc{J}_0$ is not $F_\sigma$ e.g. because the {\em almost disjointness number} of $\mc{J}_0$,
$\mf{a}(\mc{I}_0)=\om$ (simply $\{X_k\colon k\in\om\}$ is an $\mc{J}_0$-MAD family), and we know that $\mf{a}(\mc{J})>\om$ for every $F_\sigma$ ideal $\mc{J}$ (for more details see e.g. \cite{Farkas-Soukup}). We will come back to the question of
representability in $\mathbb{R}^\om$ later (see Question \ref{R^om}).
\end{exa}

\begin{prop}$ $
\begin{itemize}
\item[(a)] Every ideal $\mc{J}$ on $\om$ is representable in a normed space.
\item[(b)] There is a normed space $X$ with $\mathrm{dim}(X)=2^\mf{c}$ such that every $\mc{J}$ is representable in $X$.
\item[(c)] Every ideal $\mc{J}$ is representable in a group satisfying $g+g=0$.
\end{itemize}
\end{prop}
\begin{proof}
(a): Let $X_\mc{J}$ be the linear subspace of $\ell_\infty$ (or $\ell_1$ or $c_0$) generated by
\[ \bigg\{\bigg(\frac{\chi_A(n)}{n^2}\bigg)_{n\in\om}:A\in \mc{J}\bigg\}\]
where $\chi_A$ is the characteristic function of $A$, $e_n=(\delta_{n,k})_{k\in\om}$. Let $h\colon \om\to X$, $h(n)=n^{-2}e_n$.

We claim that $\mc{J}=\mc{I}^{X_\mc{J}}_h$. Clearly $\mc{J}\subseteq\mc{I}^{X_\mc{J}}_h$. Conversely, if $B\notin\mc{J}$ and $A_0,\dots,A_{k-1}\in\mc{J}$ then we can pick an $m\in B\setminus (A_0\cup\dots\cup A_{k-1})$ and hence
\[ \Big\|\sum_{n\in B}h(n)-\al_0\sum_{n\in A_0}h(n)-\dots-\al_{k-1}\sum_{n\in A_{k-1}}h(n)\Big\|\geq \frac{1}{m^2}\]
for any $\al_0,\dots,\al_{k-1}\in\mbb{R}$ which yields that $\sum_{n\in B}h(n)=\big(\frac{\chi_B(n)}{n^2}\big)_{n\in\om}\notin X_\mc{J}$.

\smallskip
(b): Let $\mathrm{ID}$ be the family of all ideals on $\om$ (we know that $|\mathrm{ID}|=2^\mf{c}$) and let $X$ be the finite support product of $X_\mc{J}$'s:
\[ X=\bigoplus_{\mc{J}\in\mathrm{ID}}X_\mc{J}=\bigg\{x\in\prod_{\mc{J}\in\mathrm{ID}}X_\mc{J}:\big\{\|x(\mc{J})\|\ne 0:\mc{J}\in\mathrm{ID}\big\}\;\;\text{is finite}\bigg\}\]
with the norm $\|x\|=\sup\big\{\|x(\mc{J})\|:\mc{J}\in\mathrm{ID}\big\}$.
Clearly, if $\mc{J}\ne\mathrm{Fin}$ then $\dim(X_\mc{J})=\mf{c}$  and hence $\dim(X)=2^\mf{c}$.

\smallskip
(c): Equip $\mc{J}$ with the subspace topology inherited from $\mc{P}(\om)$ and with the usual group operation (symmetric difference). Then it is easy to see that $\mc{J}=\mc{I}^\mc{J}_h$ where $h(n)=\{n\}$.
\end{proof}

\begin{que}
Does there exist a normed space $X$ such that all ideals on $\om$ are representable in $X$ but $\dim(X)<2^\mf{c}$? (If $2^\mf{c}=\mf{c}^{+n}$ for some $n\in\om$, then the answer is NO because in this case $|X|^\om<2^\mf{c}$.)
\end{que}

\section{Characterization of Polish- and Banach-representability} \label{main}

\begin{thm}\label{polishrep}
An ideal $\mc{J}$ is Polish-representable if and only if $\mc{J}$ is an analytic P-ideal.
\end{thm}
\begin{proof}
We present two proofs for the ``only if'' part. In the first one, we show that Polish-representable ideals are $F_{\sigma\delta}$ P-ideals by a direct calculation. In the second proof we show that all Polish-representable ideals are of the form $\mathrm{Exh}(\varphi)$ for some lsc submeasure $\varphi$.

First proof (sketch): Let $G$ be a Polish Abelian  group, $d$ be a complete and translation invariant (compatible) metric on $G$, and assume that $h\colon \om\to G$ such that $\sum_{n\in\om}h(n)$ does not exist. Then
\begin{align*} \mc{I}^G_h & =\Big\{A\subseteq\om\colon \;\text{the net}\;\;\sum h\!\upharpoonright\! A\;\;\text{is Cauchy}\Big\}\\
& =\bigcap_{\varepsilon>0}\bigcup_{F\in [\om]^{<\om}}\bigcap_{E\in [\om\setminus F]^{<\om}}\Big\{A\subseteq\om:d\big(0,s_h(A\cap E)\big)<\varepsilon\Big\}\end{align*}
and the last set is clearly clopen, hence $\mc{I}^G_h$ is $F_{\sigma\delta}$.

Next, we show that $\mc{I}^G_h$ is a P-ideal.
Let $(A_k)_{k\in\om}$ be a sequence of pairwise disjoint elements of $\mc{I}^G_h$, $\sum_{n\in A_k}h(n)=a_k$. For every $k$ we can choose an $N_k\in \om$ such that if $E\in [A_k\setminus N_k]^{<\om}$, then $d(0,s_h(E))<2^{-k}$. Clearly, $b_k:=\sum_{n\in A_k\setminus N_k}h(n)=a_k-s_h(A\cap N_k)$ and $d(0,b_k)\leq 2^{-k}$.
It is not difficult to see that $A=\bigcup_{k\in\om}A_k\setminus N_k\in\mc{I}^G_h$ and of course $A_k\subseteq^* A$ for every $k$. The only additional property of $d$ we need to use here is the following  easy consequence of the translation invariance: $d(0,g_0+g_1+\dots+g_{n-1})\leq d(0,g_0)+d(0,g_1)+\dots+d(0,g_{n-1})$ for $g_0,g_1,\dots,g_{n-1}\in G$.

\smallskip
Second proof: We show that if $G$ is a Polish Abelian  group and $h\colon \om\to G$, then there is an lsc submeasure $\varphi$ such that $\mc{I}^G_h=\mathrm{Exh}(\varphi)$.
Let $\varphi$ be defined by $\varphi(\emptyset)=0$ and if $A\ne\emptyset$ then  \[\varphi(A)=\sup\big\{d\big(0,s_h(F)\big):\emptyset\ne F\in [A]^{<\om}\big\},\]
where $d$ is a complete and translation invariant metric on $G$. Applying translation invariance of $d$ (see above), it is easy to see that $\varphi$ is an lsc submeasure.

\smallskip
$\mc{I}^G_h\subseteq\mathrm{Exh}(\varphi)$: Assume that $A\in\mc{I}^G_h$, i.e. that $\sum h\!\upharpoonright\! A$ is Cauchy, that is, for every $\eps>0$ there is an $N\in\om$ such that $d(0,s_h(E))<\eps$ for every $E\in [A\setminus N]^{<\om}$. Then $\varphi(A\setminus N)\leq\eps$ so $\lim_{N\to\infty}\varphi(A\setminus N)=0$.

\smallskip
$\mathrm{Exh}(\varphi)\subseteq\mc{I}^G_h$: Assume that $A\in\mathrm{Exh}(\varphi)$, that is, $\varphi(A\setminus N)\to 0$ if $N\to\infty$. Assume that $\varphi(A\setminus N)<\eps$. If $E\in [A\setminus N]^{<\om}$ then $d(0,s_h(E))\leq \varphi(A\setminus N)<\eps$. It yields that $\sum h\rest A$ is Cauchy, i.e. $A\in\mc{I}^G_h$.

\medskip
Proof of the ``if'' part: Let $\mc{J}=\mathrm{Exh}(\varphi)$ be an analytic P-ideal. We show that $\mc{J}=\mc{I}^{\mathrm{Exh}(\varphi)}_h$ where $h\colon \om\to \mathrm{Exh}(\varphi)$, $h(n)=\{n\}$.
First assume that $A\in\mc{J}$. If $\varphi(A\setminus n)<\eps$,
then $d_\varphi(A,E)<\eps$ whenever $A\cap n\subseteq E\in [A]^{<\om}$ (of course, $s_h(E)=E$) hence $\sum_{n\in A}h(n)=A\in\mc{I}^{\mathrm{Exh}(\varphi)}_h$.
If $B\notin \mc{J}$, $A\in\mc{J}$, and $n_0\in B\setminus A$, then $d_\varphi(A,E)\geq\varphi(\{n_0\})$ for every $E\in [B]^{<\om}$ with $n_0\in E$, in other words $\sum_{n\in B}h(n)\ne A$, and so $B\notin \mc{I}^{\mathrm{Exh}(\varphi)}_h$.
\end{proof}

To characterize Banach-representability, we need the following notion:

\smallskip
An lsc submeasure $\varphi$ is {\em non-pathological} if it is the (pointwise) supremum of measures dominated by $\varphi$, i.e. for every $A\subseteq\om$
\[ \varphi(A)=\sup\big\{\mu(A)\colon \mu\;\;\text{is a measure on}\;\;\mc{P}(\om),\;\;\text{and}\;\;\forall\;B\subseteq\om\;\mu(B)\leq\varphi(B)\big\}.\]
Because of the lower semicontinuity, it is enough to check that this equality holds for every $A\in\mathrm{Fin}$ (for more details and characterizations of non-pathological submeasures see  \cite[Cor. 5.26]{Hrusak}).

An analytic P-ideal $\mc{J}$ is {\em non-pathological} iff $\mc{J}=\mathrm{Exh}(\varphi)$ for some non-pathological lsc submeasure $\varphi$. For example, summable ideals, density ideals, Farah's ideal, $\mathrm{tr}(\mc{N})$, and Tsirelson ideals are non-pathological.
In general, constructions of pathological ideals, even pathological
lsc submeasures are 
non-trivial (see \cite{Mazur} for an example of such a construction and \cite{farah} for further references).

It is easy to see that all non-pathological generalized
density ideals are representable in $c_0$ (see the idea of Example \ref{Zinc0}).

\smallskip
If $h$ is a function from $\om$ to a classical sequence space, then we will write $h=(x^n_k)_{n,k\in\om}$ if $h(n)=(x^n_k)_{k\in\om}$. If $x^n_k\geq 0$ for all $n,k$, then we write $h\geq 0$.

\begin{lem}\label{absval}
Assume that $h=(x^n_k)_{n,k\in\om}\colon \om\to\ell_\infty$ is such that $\sum h$ does not converge. If $h'=(|x^n_k|)_{n,k\in\om}$ then $\mc{I}^{\ell_\infty}_h=\mc{I}^{\ell_\infty}_{h'}$.
\end{lem}
\begin{proof}
$\mc{I}^{\ell_\infty}_{h'}\subseteq\mc{I}^{\ell_\infty}_h$: Trivial because $\|s_h(F)\|\leq \|s_{h'}(F)\|$ for every finite $F\subseteq\om$.
$\mc{I}^{\ell_\infty}_h\subseteq\mc{I}^{\ell_\infty}_{h'}$: Assume $\sum h'\rest A$ is not convergent, i.e. not Cauchy. It means that for some $\eps>0$ for all $N\in\om$ there is an $F_N\in [A\setminus N]^{<\om}$ such that $\|s_{h'}(F_N)\|\geq\eps$.
We show that neither $\sum h\rest A$ is Cauchy and that $\frac{\eps}{4}$ witnesses it. Indeed, let $N$ be arbitrary and fix a $k$ such that \[ \big|\sum_{n\in F_N}|x^n_k|\big|=\sum_{n\in F_N}|x^n_k|>\frac{\eps}{2}.\] Let $F_N=F^0_N\cup F^1_N$ be a
partition such that $k\in F^0_N$ iff $x^n_k<0$. Then \[  \|s_h(F^0_N)\|=|\sum_{n\in F^0_N}x^n_k|>\frac{\eps}{4} \mbox{\ \ \ \ \ \ or\ \ \ \ \ \  } \|s_h(F^1_N)\|=|\sum_{n\in F^1_N}x^n_k|>\frac{\eps}{4}.\]
\end{proof}

\begin{thm}\label{B-rep}
An analytic P-ideal $\mc{J}$ is Banach-representable iff it is non-pathological.
\end{thm}
\begin{proof}
Proof of the ``if'' part: Let $\mc{J}=\mathrm{Exh}(\varphi)$ for some non-pathological $\varphi$, and let $(\mu_k)_{k\in\om}$ be a sequence of measures on $\om$ such that $\varphi(F)=\sup\big\{\mu_k(F)\colon k\in\om\big\}$ holds for every $F\in
[\om]^{<\om}$, and so for every subsets of $\omega$. Let $h\colon \om\to \ell_\infty$ be defined by  $h=(\mu_k(\{n\}))_{n,k\in\om}$, i.e. \[ h(n)=\big(\mu_0(\{n\}),\mu_1(\{n\}),\mu_2(\{n\}),\dots\big)\] (confront Example \ref{Zinc0}).
Clearly, if $F\in [\om]^{<\om}$ then $s_h(F)=\big(\mu_0(F),\mu_1(F),\dots)$ and hence $\|s_h(F)\|=\varphi(F)$. It implies that $h$ is as required because
\begin{align*}
A\in\mc{I}^{\ell_\infty}_h\;\;\;\text{iff} &\;\;\;\sum h\rest A\;\;\text{is convergent (i.e. Cauchy)}\\
\text{iff} & \;\;\;\forall\;\eps>0\;\exists\;N\;\forall\;F\in [A\setminus N]^{<\om}\;\;\|s_h(F)\|=\varphi(F)\leq\eps\\
\text{iff} & \;\;\;\forall\;\eps>0\;\exists\;N\;\;\varphi(A\setminus N)\leq\eps\;\;\text{(i.e.}\;\lim_{n\to\infty}\varphi(A\setminus n)=0\text{)}\\
\text{iff} & \;\;\;A\in\mathrm{Exh}(\varphi).\end{align*}

\smallskip
Proof of the ``only if'' part: Assume $\mc{J}=\mc{I}^{X}_h$ for some Banach space $X$ and $h\colon \om\to X$. We can assume that $X=\ell_\infty$ because $\ell_\infty$ contains isometric copies of all separable Banach spaces. Applying Lemma
\ref{absval}, we can also assume that $h\geq 0$. For $k\in\om$ and $A\subseteq\om$ let $\mu_k(A)=\sum_{n\in A}h(n)_k$ (so $\mu_k$ is a measure on $\om$), and let $\varphi=\sup\{\mu_k\colon k\in\om\}$. Just like in the proof of the ``if'' part, we have
$\|s_h(F)\|=\varphi(F)$ and, by the same argument as above, one can prove that $\mc{I}^{\ell_\infty}_h=\mathrm{Exh}(\varphi)$.
\end{proof}

\begin{rem} We would like to present an alternative proof for the ``only if'' part where we do not need to use $\ell_\infty$ (and Lemma \ref{absval}). Assume that $\mathrm{Exh}(\varphi)=\mc{I}^X_h$ for some Banach space $X$ and $h\colon \om\to X$. We will construct a non-pathological  lsc submeasure $\psi$ such that $\mathrm{Exh}(\varphi)=\mathrm{Exh}(\psi)$.
	
Let $\widetilde{\varphi}\colon \mc{P}(\om)\to [0,\infty]$ be defined by $\widetilde{\varphi}(\emptyset)=0$ and for $A\ne\emptyset$
\[ \widetilde{\varphi}(A)=\sup\big\{\|s_h(F)\|\colon \emptyset\ne F\in [A]^{<\om}\big\}.\]
In Theorem \ref{polishrep} we already proved that $\widetilde{\varphi}$ is an lsc submeasure and  $\mathrm{Exh}(\widetilde{\varphi})=\mc{I}^X_h=\mathrm{Exh}(\varphi)$.
How to construct $\psi$? Fix an $F\in [\om]^{<\om}$ and let $F'\subseteq F$ such that $\widetilde{\varphi}(F)=\|s_h(F')\|$. Applying the Hahn-Banach theorem, there is an $x_F^*\in X^*$ with $\|x_F^*\|=1$ such that
$x_F^*\big(s_h(F')\big)=\|s_h(F')\|$. Then the function $\nu_F\colon \mc{P}(\om)\to [0,\infty]$
\[ \nu_F(B)=x_F^*\big(s_h(F'\cap B)\big)=\sum_{n\in F'\cap B}x_a^*(h(n))\]
defines a signed measure with support $F'$. If $\nu_F=\nu_F^+-\nu_F^-$ where $\nu_F^+,\nu_F^-$ are measures and $\nu_F^+\perp\nu_F^-$, then let $\mu_F=\nu_F^++\nu_F^-$. (In other words, the measure $\mu_F$ is uniquely determined by $\mu_F(\{n\})=|\nu_F(\{n\})|$.) Finally let $\psi=\sup\{\mu_F:F\in [\om]^{<\om}\}$, a non-pathological lsc submeasure.

We claim that $\widetilde{\varphi}\leq\psi\leq 2\widetilde{\varphi}$ and hence $\mathrm{Exh}(\psi)=\mathrm{Exh}(\widetilde{\varphi})=\mathrm{Exh}(\varphi)$.

$\widetilde{\varphi}\leq\psi$: $\widetilde{\varphi}(F)=\|s_h(F')\|=x^*_F(s_h(F'))=\nu_F(F')=\nu_F(F)\leq\mu_F(F)\leq\psi(F)$.

$\psi\leq 2\widetilde{\varphi}$: for every finite $F$ and $E$ if $P=\mathrm{supp}(\nu_F^+)=\{k\in F':x_F^*(h(k))> 0\}$ then
\begin{align*}
\mu_F(E)= & \mu_F(F'\cap E)=\nu_F^+(F'\cap E)+\nu_F^-(F'\cap E)\\
= & \big|x^*_F\big(s_h(F'\cap E\cap P)\big)\big|+\big|x^*_F\big(s_h(F'\cap E\setminus P)\big)\big|\\
\leq & \big\|s_h(F'\cap E\cap P)\big\|+\big\|s_h(F'\cap E\setminus P)\big\|\\
\leq & \widetilde{\varphi}(E)+\widetilde{\varphi}(E),
\end{align*}
where we used that $|x^*_F(y)|\leq \|y\|$ for every $y\in X$ because $\|x^*_F\|=1$.
\end{rem}

\section{Representability in $c_0$} \label{c0}

We will need the following notions (see \cite{SoTo} and \cite{Matrai}):

An lsc submeasure $\varphi$ is {\em density-like} if for every $\eps>0$ there is a $\de>0$ such that if $A_n\in [\om]^{<\om}$ is a sequence of pairwise disjoint finite sets with $\varphi(A_n)<\de$, then there is an $X\in [\om]^\om$ such that
$\varphi\big(\bigcup_{n\in X}A_n\big)<\eps$. An analytic P-ideal $\mc{J}$ is {\em density-like} if there is a density-like submeasure $\varphi$ such that $\mc{J}=\mathrm{Exh}(\varphi)$. Clearly, generalized density ideals are density-like. At this moment we do not have any other examples of density-like ideals (see Question \ref{densitylike}).

An lsc submeasure $\varphi$ is {\em summable-like} if there is an $\eps>0$ such that for every $\de>0$ there is a sequence $A_n\in [\om]^{<\om}$ of pairwise disjoint finite sets with $\varphi(A_n)<\de$ and there is a $k\in \om$ such that $\varphi\big(\bigcup_{n\in Y}A_n\big)\geq\eps$ for every $Y\in [\om]^k$. An analytic P-ideal $\mc{J}$ is {\em summable-like} if there is a summable-like submeasure $\varphi$ such that $\mc{J}=\mathrm{Exh}(\varphi)$. For example, summable ideals which are not trivial modifications of $\mathrm{Fin}$ (i.e. $\ne\{A\subseteq\om:|A\cap X|<\om\}$ for some $X\subseteq\om$) and Farah's ideal are summable-like.

Applying Remark \ref{pleb}, it is not difficult to see that if $\mathrm{Exh}(\varphi)=\mathrm{Exh}(\psi)$ and $\varphi$ is summable- / density-like, then $\psi$ is also summable- / density-like.

Clearly, an ideal cannot be both density- and summable-like. Moreover, tall $F_\sigma$ P-ideals are not density-like. In \cite{Matrai} M\'{a}trai constructed an $F_\sigma$ P-ideal which is neither of them. S\l awek Solecki remarked (in personal
communication) that the Tsirelson ideal defined through the classical Tsirelson space (not its dual) is another example of $F_\sigma$ P-ideal which is neither summable- nor density-like (and is clearly non-pathological).

A less obvious example of summable-like ideal is $\mathrm{tr}(\mc{N})$ (see below) which is interesting because in some sense it is as far from being a real summable ideal as it is possible: in \cite{HHH} is was proved that $\mathrm{tr}(\mc{N})$ and $\mc{Z}$  are
{\em totally bounded}, that is, $\varphi$ must be finite (i.e. $\varphi(\om)<\infty$) whenever $\mathrm{tr}(\mc{N})=\mathrm{Exh}(\varphi)$ (or $\mc{Z}=\mathrm{Exh}(\varphi)$). The authors of \cite{HHH} observed that if the {\em splitting number}, $\mf{s}(\mc{J})$ of an analytic P-ideal $\mc{J}$ is $\omega$ then it must be totally bounded.

\begin{prop}
$\mathrm{tr}(\mc{N})$ is summable-like.
\end{prop}
\begin{proof}
We know that $\mathrm{tr}(\mc{N})=\mathrm{Exh}(\varphi)$ where
\[ \varphi(A)=\sup\bigg\{\sum_{s\in B}2^{-|s|}\colon B\subseteq A\;\;\text{is an antichain}\bigg\}.\]
Let $\eps=\frac{1}{2}$ and $\de>0$ be arbitrary. Fix an $m\in\om$ such that $2^{-m}<\de$, and for every $n$ let \[ A_n=\big\{s\in 2^{nm+m}:s\upharpoonright [nm,nm+m)\equiv 0\big\}.\]
It is easy to see that $A_n$ is a finite antichain and that the measure of the associated clopen set $\widetilde{A}_n=\bigcup_{s\in A_n} [s]$ is $\varphi(A_n)=2^{-m}<\de$ where $[s]=\{x\in 2^\om\colon s\subseteq x\}$. Clearly, $\varphi(A)=\lam(\widetilde{A})$.

The family $\{\widetilde{A}_n\colon n\in\om\}$ forms an independent system: if $n_0<n_1<\dots<n_{k-1}$, then
\[ \widetilde{A}_{n_0}\cap\widetilde{A}_{n_1}\cap\dots\cap\widetilde{A}_{n_{k-1}}=\big\{x\in 2^\om\colon \forall i<k \ \ x\upharpoonright[n_im,(n_i+1)m)\equiv 0\big\}\] and hence $\lam\big(\bigcap_{i<k}\widetilde{A}_{n_i}\big)=2^{-mk}=(2^{-m})^k$.

Applying independence, if $Y\in [\om]^{2^m}$ then
\[ \varphi\bigg(\bigcup_{n\in Y}A_n\bigg)=\lam\bigg(\bigcup_{n\in Y}\widetilde{A}_n\bigg)=\sum_{k=1} ^{2^m}(-1)^{k+1}\binom{2^m}{k}\big(2^{-m}\big)^k=1-(1-2^{-m})^{2^m}\xrightarrow{m\to\infty}1-\frac{1}{\mathrm{e}}\]
therefore if $m$ is large enough, then $\varphi\big(\bigcup_{n\in Y}A_n\big)>\frac{1}{2}=\eps$.
\end{proof}

Representability in $c_0$ can be characterized by combinatorics of the defining submeasure.
This approach will help us showing that several classical ideals are not representable in $c_0$.

\begin{prop}\label{Simple=c0}
An ideal $\mc{J}$ is representable in $c_0$ iff there is a lsc submeasure $\varphi$ and a sequence $(\mu_k)_{k\in\om}$ of measures on $\om$ such that $\mc{J}=\mathrm{Exh}(\varphi)$, $\varphi=\sup\{\mu_k\colon k\in\om\}$, and $\{k\colon m\in\mathrm{supp}(\mu_k)\}$ is finite for every $m\in\om$.
\end{prop}
\begin{proof}
If $\mc{J}=\mathrm{Exh}(\varphi)$ for some submeasure $\varphi=\sup\{\mu_k\colon k\in\om\}$ as in the statement. Then the basic representation of $\mathrm{Exh}(\varphi)$ in $\ell_\infty$ (see Theorem \ref{B-rep}) is actually a representation in $c_0$.

Now assume that $h=(x^n_k)\colon \om\to c_0$, $h\geq 0$, and $\mc{J}=\mc{I}^{c_0}_h$. We will modify $h$. To every $n$ fix a $k_n$ such that $|x^n_k|<2^{-n}$ for every $k\geq k_n$, and let $h'=(y^n_k)\colon \om\to c_0$ where $y^n_k=x^n_k$ if
$k<k_n$ (otherwise $y^n_k=0$). It is easy to see that $\mc{I}^{c_0}_{h'}=\mc{J}$ and therefore we can use the proof of Theorem \ref{B-rep} again to obtain the measures $\{\mu_k\colon k\in\om\}$ such that $\mc{J}=\mc{I}^{c_0}_{h'}=\mathrm{Exh}(\psi)$
where $\psi=\sup\{\mu_k\colon k\in\om\}$. Clearly, $\psi$ is as desired.
\end{proof}

\begin{prop}\label{Verysimple}  Let $(\mu_k)_{k\in\om}$ be a sequence of measures on $\om$ such that $\{k\colon m\in\mathrm{supp}(\mu_k)\}$ is finite for every $m\in\om$. Let $\varphi=\sup\{\mu_k\colon k\in\om\}$ and  $\mc{J}=\mathrm{Exh}(\varphi)$. If $\mu_k$ is bounded for every $k$, then $\mc{J}$ is a generalized density ideal.
\end{prop}
\begin{proof}
	 For any $k$ we can fix an $n_k$ such that $\mu_k(\om\setminus n_k)<2^{-k}$. Let $\mu'_k(A)=\mu_k(A\cap n_k)$. We claim that if
	$\varphi'=\sup\{\mu'_k\colon k\in\om\}$, then $\mathrm{Exh}(\varphi')=\mathrm{Exh}(\varphi)$.

Clearly, $\mathrm{Exh}(\varphi')\supseteq\mathrm{Exh}(\varphi)$ (because $\varphi'\leq \varphi$).
So, assume that $A\in\mathrm{Exh}(\varphi')$, i.e. for every $\eps>0$ there is an $N\in\om$ such that if $F\in [A\setminus N]^{<\om}$, then $\varphi'(F)<\eps$. We will find an $M$ such that $\varphi(F)<2\eps$ for every $F\in [A\setminus M]^{<\om}$. Let
$K\in\om$ be such that $2^{-K-1}\leq\eps<2^{-K}$. For all $k\leq K$ fix an $m_k\geq n_k$ such that $\mu_k(\om\setminus m_k)<2^{-K-1}$ and let $M=\max\{N,m_0,m_1,\dots,m_K\}$. It is easy to see that if $F\in [A\setminus M]^{<\om}$, then $\varphi(F)< \varphi'(F)+2^{-K-1}<2\eps$.

To finish the proof, we show the following general fact.

\begin{claim} Assume that $\psi=\sup\{\nu_k\colon k\in\om\}$ where $\nu_k$ is a measure for every $k$, $|\{k\colon m\in\mathrm{supp}(\nu_k)\}|<\om$ for every $m$, and $|\mathrm{supp}(\nu_k)|<\om$ for every $k$. Then $\mathrm{Exh}(\psi)$ is a generalized density ideal.
\end{claim}
\begin{proof}
We can easily find an interval partition $(P_n)_{n\in\om}$ of $\om$ such that for every $k$ there is an $n$ with $\mathrm{supp}(\nu_k)\subseteq P_n\cup P_{n+1}$. Let $\varphi_n(A)=\sup\{\nu_k(A\cap P_n)\colon k\in\om\}$ for every $n$. Notice that $\varphi_n$ is a
submeasure concentrated on $P_n$. We show that if $\vec{\varphi}=(\varphi_n)_{n\in\om}$, then $\mc{Z}_{\vec{\varphi}}=\mathrm{Exh}(\psi)$.

Clearly $\sup_{n\in\om}\varphi_n(A)\leq\sup_{k\in\om}\nu_k(A)$ holds for every $A$, in particular \[ \varphi_n(A)=\varphi_n(A\cap P_n)\leq\sup_{k\in\om}\nu_k(A\setminus \min(P_n)),\] and thus $\mc{Z}_{\vec{\varphi}}\supseteq \mathrm{Exh}(\psi)$.
Conversely, if $A\notin \mathrm{Exh}(\psi)$  then there is an $\eps>0$ such that for every $m$ there is a $k_m$ such that $\nu_{k_m}(A\setminus m)>\eps$. The set $\{k_m\colon m\in\om\}$ is infinite because the supports of $\nu_k$'s are finite. For a
fixed $m$, there is an $n_m$ such that $\mathrm{supp}(\nu_{k_m})\subseteq P_{n_m}\cup P_{n_m+1}$ and hence \[ \varphi_{n_m}(A)\geq\varphi_{n_m}(A\setminus m)>\eps/2\mbox{\ \ \ \ \ \ or\ \ \ \ \ \ }\varphi_{n_m+1}(A)\geq\varphi_{n_m+1}(A\setminus
m)>\eps/2.\] The set $\{n_m\colon
m\in\om\}$ is also infinite because $P_n\cup P_{n+1}$ can cover only finitely many supports of $\nu_k$'s. Therefore $\varphi_n(A)\nrightarrow 0$ and $A\notin\mc{Z}_{\vec{\varphi}}$.
\end{proof}

Clearly, $\varphi'$ and the sequence $(\mu'_k)_k$ satisfies the conditions of the claim, so we are done.
\end{proof}

\begin{cor}
If an analytic P-ideal is totally bounded and representable in $c_0$, then it is a generalized density ideal.
\end{cor}

\begin{cor}
$\mathrm{tr}(\mc{N})$ is not representable in $c_0$.
\end{cor}

The next result shows that among tall $F_\sigma$ ideals the representability in $c_0$ is equivalent to the representability in $\mathbb{R}$.

\begin{thm}\label{F-sigma}
A tall $F_\sigma$ P-ideal $\mathcal{I}$ is representable in $c_0$ iff it is a summable ideal.
\end{thm}

\begin{proof} (Non-trivial implication.)
Fix an lsc submeasure $\varphi$ such that $\mathcal{I} = \mathrm{Exh}(\varphi) = \mathrm{Fin}(\varphi)$. Suppose that $\mc{I}$ is representable in $c_0$, and (using Proposition \ref{Simple=c0}) fix a non-pathological submeasure
	$\psi=\sup\{\mu_n\colon n\in\om\}$ such that $\mc{I}=\mathrm{Exh}(\psi)$ and there is a strictly increasing function $f\colon \om\to\om$ such that $\mu_n(\{k\})=0$ if $n\geq f(k)$. Also, assume that $\varphi\geq \psi$. Otherwise, we could consider $\varphi+\psi$ instead of $\varphi$ (notice that $\mathrm{Exh}(\varphi)=\mathrm{Exh}(\varphi+\psi)$ and $\mathrm{Fin}(\varphi) =\mathrm{Fin}(\varphi+\psi)$). Because of tallness, $\varphi(\{n\}),\psi(\{n\})\to 0$.

\smallskip
Assume on the contrary that $\mathcal{I}$ is not summable. Let $\psi_n = \max_{m\leq n} \mu_m$. Then $\mathcal{I}_n = \mathrm{Exh}(\psi_n)$ is summable for every $n$ ($\mc{I}_n=\mc{I}_{h_n}$ where $h_n(k)=\psi_n(\{k\})$) and we have $\mathcal{I} \subseteq\dots\subseteq \mathcal{I}_{n+1}\subseteq \mc{I}_n\subseteq\dots\subseteq\mc{I}_0$. Hence, for each $n$ we can find $A_n \in \mathcal{I}_n\setminus \mathcal{I}$. We can assume that (i) $(A_n)_{n\in\om}$ are pairwise disjoint and (ii) $\psi_n(A_n)<2^{-n}$ for every $n$.

(i): Fix a sequence $(B_k)_{k\in\om}$ such that $H_n=\{k\colon A_n=B_k\}$ is infinite for every $n$. Then by recursion we can pick finite sets $F_k\subseteq B_k$ such that $\varphi(F_k)>k$ and $\max(F_k)<\min(F_{k+1})$. Finally let $A'_n=\bigcup\{F_k:k\in H_n\}$. Then $A'_n\subseteq A_n$ and $A'_n\notin \mc{I}$.

(ii): Applying tallness, for every $n$ there is an $m_n$ such that \[ \psi_n(A_n\setminus m_n)\leq\sum_{m\in A_n\setminus m_n}\psi_n(\{m\})<2^{-n},\]
hence after finite modifications (ii) holds true.

In particular, $\bigcup_{k\geq n}A_k\in\mc{I}_n$ because $\psi_n(A_k)\leq \psi_k(A_k)<2^{-k}$ (if $n\leq k$) and hence we can use $\sigma$-subadditivity of $\psi_n$.

Now we will construct a set $X\in \mathrm{Exh}(\psi)\setminus\mathrm{Fin}(\varphi)$, which will lead us to the contradiction. First, we can fix a sequence $(X'_n)_{n\in\om}$ such that
\begin{enumerate}
	\item[(a)] $X'_0 \in [A_0]^{<\om}$ and $X'_{n+1}\in \big[A_{f(\max(X'_n)+1)}\big]^{<\om}$,
	\item[(b)] $\max(X'_n)<\min(X'_{n+1})$,
	\item[(c)] $\varphi(X'_n)\approx 1$ for every $n$.
\end{enumerate}
Let $X'=\bigcup_{n\in\om}X'_n$ and consider the following ideals on $X'$: $\mc{I}\upharpoonright X'$ is a tall $F_\sigma$ P-ideal (hence not density-like), and $\mc{Z}_{\vec{\varphi}}$ is a tall generalized density ideal (hence density-like) where
$\vec{\varphi}=(\varphi\upharpoonright X'_n)_{n\in\om}$. Clearly, $\mc{I}\upharpoonright X' \subseteq\mc{Z}_{\vec{\varphi}}$. Therefore there is an $X\subseteq X'$ such that $\varphi(X)=\infty$ (i.e. $X\notin\mathrm{Fin}(\varphi)$) but $\varphi(X\cap
X'_n)\to 0$. We show that $X\in\mathrm{Exh}(\psi)$.

Let $k_n=\max(X'_n)+1$ and $X_n=X\cap X'_n$. Clearly $X_{n+1}\subseteq A_{f(k_n)}\cap [k_n,k_{n+1})$.
Fix $m,N\in\om$. We have two cases:

1. $f(k_N)\geq m$. Then using $N\leq k_N\leq f(k_N)$ we obtain that
\begin{align*}
\mu_m(X\setminus k_N) & \leq \mu_m \bigg(\bigcup_{j\geq f(k_N)} A_j\bigg) \leq \psi_m \bigg(\bigcup_{j\geq f(k_N)} A_j\bigg)\\
& \leq\sum_{j\geq f(k_N)} 2^{-j}=2^{-f(k_N)+1}\leq 2^{-k_N+1}\leq 2^{-N+1}.\end{align*}

2. $f(k_N)<m$. Then we work with a partition of $X\setminus k_N$ and obtain that
\begin{itemize}
	\item $\mu_m\big(X\cap [k_N, k_{N+1})\big) = \mu_m (X_{N+1}) \leq \psi (X_{N+1}) \leq \varphi(X_{N+1})$;
	\item $\mu_m(X\setminus k_{N+1}) \leq \mu_m \big(\bigcup_{j\geq f(k_{N+1})} A_j\big) \leq \mu_m \big(\bigcup_{j\geq N} A_j\big) \leq \psi_m \big(\bigcup_{j\geq N} A_j\big)\leq 2^{-N+1}$.
\end{itemize}

In both cases, $\mu_m(X\setminus k_N) \leq \varphi(X_{N+1}) + 2^{-N+1}$. Moreover this value does not depend on $m$ and tends to $0$. Hence $\psi(X\setminus k_N)=\sup_{m\in\om}\mu_m(X\setminus k_N)\xrightarrow{N\to\infty}0$, i.e. $X\in \mathrm{Exh}(\psi)$.
\end{proof}

\begin{cor} Farah's ideal and the Tsirelson ideals are not representable in $c_0$.
\end{cor}

\begin{prop}\label{Dichotomy}
Let $\mathcal{I}$ be representable in $c_0$. Assume that there is no $A\in[\om]^\om$ such that $\mathcal{I}\upharpoonright A$ is contained in a summable ideal. Then $\mathcal{I}$ is a generalized density ideal.
\end{prop}

\begin{proof}
	 We can assume that $\mc{I}=\mathrm{Exh}(\varphi)$, where $\varphi = \sup\{\mu_k\colon k\in\om\}$ is such that $\{k\colon m\in\mathrm{supp}(\mu_k)\}$ is finite for every $m$.
	Suppose there is $k\in\omega$ such that $\mu_k$ is unbounded. Let $A = \mathrm{supp}(\mu_k)$ and notice that $\mathcal{I}\upharpoonright A \subseteq \mathrm{Exh}(\mu_k)=\mc{I}_{\mu_k(\{\cdot\})}$. Therefore, $\mathcal{I}\upharpoonright A$ is contained in a summable ideal. Hence  each $\mu_k$ is bounded. But then, by Proposition \ref{Verysimple}, $\mathcal{I}$ is a generalized density ideal.
\end{proof}



It seems that it is difficult to find a density-like ideal which is \textbf{not} a generalized density ideal (see Question \ref{densitylike}), so it is difficult to find a (non-pathological) density-like ideal which is not representable in $c_0$. The above proposition
shows that the situation differs drastically with summable-like ideals: it is impossible to find an example which \textbf{is} representable in $c_0$ and which does not resemble a summable ideal at least locally. So, representability in $c_0$ seems to be
closely connected to density-likeness. However the question of the full characterization is still open.

\begin{que}
How to characterize analytic P-ideals which can be represented in $c_0$?
\end{que}

It is connected to an irritating question mentioned before:

\begin{que}\label{densitylike}
Is there a density-like ideal which is not a generalized density ideal?
\end{que}

In the context of Proposition \ref{Dichotomy} there remains a question, how much ideals contained in a summable ideal resemble summable ideals itself. Despite the fact that summable ideals are quite small, there are nontrivial examples of ideals covered
by a summable ideal which seem to be quite far from being summable itself. A natural example of such an ideal is Farah's ideal. Indeed, notice that $\mc{J}_F$ can be covered by summable ideals in many ways, e.g. let $h(k)=1/n$ if $k\in [2^n,2^n+n)$ and $h(k)=0$ else, then $\mc{J}_F\subseteq\mc{I}_h$. There is even a density ideal covered by a summable ideal:

\begin{exa}
Let $\mu_n(\{k\})=1/n$, $h(k)=2^{-n}$ for $k\in [2^n,2^{n+1})$. Then $\mc{Z}_{\vec{\mu}}\subseteq\mc{I}_h$ simply because if $a\in\om^\om$ and $\frac{a_n}{n}\to 0$ then $\sum_{n\in\om}\frac{a_n}{2^n}<\infty$.
\end{exa}

On the other hand, $\mathrm{tr}(\mc{N})$ cannot be covered by any summable ideal.


\begin{exa}
	$\mathrm{tr}(\mc{N})$ (and hence $\mc{Z}$) cannot be covered by a summable ideal. Let $h\colon 2^{<\om}\to [0,\infty)$.  It is easy to find an branch $x\in 2^{\om}$ such that $\sum_{x\upharpoonright n \sub t} h(t) = \infty$ for each $n\in \om$.
	Choose a sequence of pairwise disjoint finite sets $(F_n)_{n\in\om}$ such that $F_n \subseteq \{t\colon x\upharpoonright n \subseteq t\}$ and $\sum_{t\in F_n}
	h(t)>1$. Then $\bigcup_{n\in\om} F_n\in\mathrm{tr}(\mc{N})\setminus\mc{I}_h$.
\end{exa}

\begin{que}\label{containedinsummable}
How do (analytic P-)ideals which are contained in a summable ideal look like?
\end{que}

This question is connected also to the question of representability in $\mathbb{R}^\omega$. We mentioned (see Example \ref{exaXk}) that an ideal is representable in $\mathbb{R}^\om$ iff it is a countable intersection of summable ideals. In particular, it has
to be contained in a summable ideal, so neither $\mathrm{tr}(\mc{N})$ nor $\mc{Z}$ is representable in $\mathbb{R}^\om$.

\begin{que}\label{R^om}
Is there any characterization of ideals representable in $\mbb{R}^\om$?
\end{que}


In the next diagram, we summarize all possible connections between properties of ideals we investigated, and also put easy examples into every ``bubble'' (where we know any).
\begin{center}
\begin{tikzpicture}[scale=1]
\draw (-8,-5) coordinate(a)
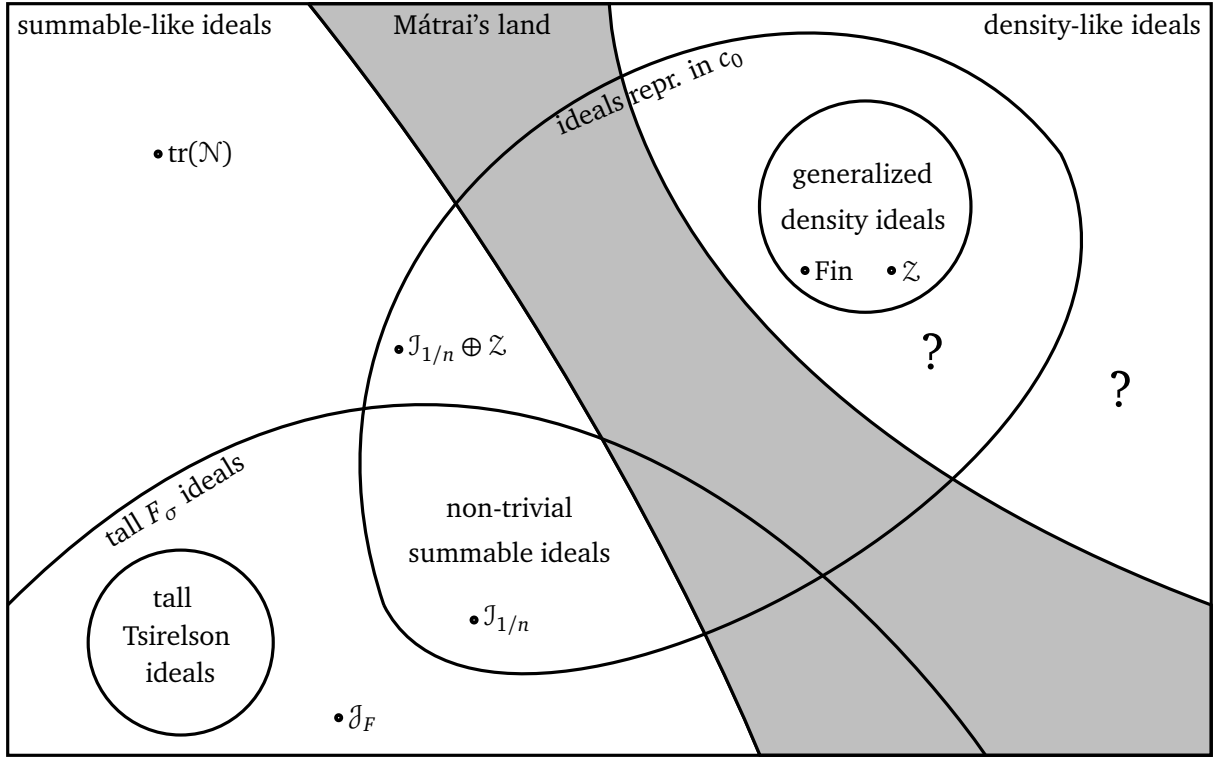
node[anchor=north west] {non-pathological analytic P-ideals (i.e. Banach-representable ideals)};
\draw (6.5,-0.5) coordinate(g)
node[anchor=south west] {{\huge ?}};
\draw (4,0) coordinate(h)
node[anchor=south west] {{\huge ?}};
\draw (-8,5) coordinate(b)
node[anchor=north west] {summable-like ideals};
\path (-4,5) coordinate(c1)
(0,0) coordinate(c2)
(2,-5) coordinate(c3);
\draw[line width=1.2pt]
(c1) .. controls (c2)
and (c3) .. (c3);
\path (0,5) coordinate(d1)
(0,0) coordinate(d2)
(8,-3) coordinate(d3);
\filldraw[fill=lightgray,draw=black,line width=1.2pt] (c1) .. controls (c2) and (c3)
.. (c3) -- (8,-5) -- (d3) ..controls (d2) and (d1) .. (d1) -- (c1);
\draw (-3,5) coordinate(m)
node[anchor=north west] {M\'atrai's land};
\draw (8,5) coordinate(x)
node[anchor=north east] {density-like ideals};
\draw[line width=1.2pt] (-8,-5) rectangle (+8,+5);
\draw (-6,3) coordinate(y)
node[anchor=west] {$\mathrm{tr}(\mathcal{N})$};
\draw[line width=1.2pt] (3.4,2.3) circle (40pt);
\draw (2.3,2.7) coordinate(yff)
node[anchor=west] {generalized};
\draw (2.15,2.1) coordinate(yfff)
node[anchor=west] {density ideals};
\draw[line width=1.4pt] (2.6,1.45) circle (1pt);
\draw (2.6,1.45) coordinate(yyyfff)
node[anchor=west] {$\mathrm{Fin}$};
\draw[line width=1.4pt] (-6,3) circle (1pt);
\draw (-3.6,-4.5) coordinate(yy)
node[anchor=west] {$\mathcal{J}_F$};
\draw[line width=1.4pt] (-3.6,-4.5) circle (1pt);
\draw (-2.8,0.4) coordinate(yy)
node[anchor=west] {$\mathcal{I}_{1/n}\oplus\mathcal{Z}$};
\draw[line width=1.4pt] (3.75,1.45) circle (1pt);
\draw (3.75,1.45) coordinate(yygf)
node[anchor=west] {$\mathcal{Z}$};

\draw[line width=1.4pt] (-1.8,-3.2) circle (1pt);
\draw (-1.8,-3.2) coordinate(yygf)
node[anchor=west] {$\mathcal{I}_{1/n}$};
\draw[line width=1.4pt] (-2.8,0.4) circle (1pt);
\path (-8,-3) coordinate(e1)
(-1,4) coordinate(e2)
(5,-5) coordinate(e3);
\draw[line width=1.2pt]
(e1) .. controls (e2)
and (e3) .. (e3)
node[pos=0.1,sloped,anchor=north] {tall $F_\sigma$ ideals};
\draw[line width=1.2pt] (-3,-3) ..controls (-5,3) and (3,7) .. (6,3)
node[pos=0.56,sloped,anchor=north] {ideals repr. in $c_0$} ..controls (8,-1) and (-1.5,-6) .. (-3,-3);


\draw[line width=1.2pt] (-5.7,-3.5) circle (35pt);
\draw (-6.2,-2.9) coordinate(frfrf)
node[anchor=west] {tall};
\draw (-6.6,-3.4) coordinate(fr)
node[anchor=west] {Tsirelson};
\draw (-6.3,-3.9) coordinate(frr)
node[anchor=west] {ideals};

\draw (-2.3,-1.7) coordinate(yff)
node[anchor=west] {non-trivial};
\draw (-2.8,-2.3) coordinate(yfff)
node[anchor=west] {summable ideals};
\end{tikzpicture}\,
\end{center}

Explanations:
\begin{itemize}
\item A summable ideal is {\em non-trivial} if it is not a trivial modification of $\mathrm{Fin}$.
\item $\mc{J}_0\oplus\mc{J}_1=\{A\subseteq 2\times\om\colon \{n\colon (i,n)\in A\}\in \mc{J}_i, i=0,1\}$.
\item ``?'' means that we do not know any examples in this ``bubble.''
\end{itemize}

\section{Representability in $\ell_1$}\label{ell_1}

Notice that if an ideal is represented in $\ell_1$ by a sequence whose every coordinate is non-negative, then it is a summable ideal. However, in general this is not true. We present an example of an ideal $\mathcal{I}$ which is representable in $\ell_1$ but is not summable. The construction is motivated by the standard example of an unconditionally convergent but not absolutely convergent sequence in $\ell_1$.

We will need two interval partitions $(P_n)_{n\in\om}$ and $(Q_n)_{n\in\om}$ of $\om$: Let $P_n=[0+1+2+\dots+n,0+1+2+\dots+n+(n+1))$, let $Q_0=[0,1)$ and if $n>0$ then let $Q_n=[1+2+4+\dots+2^{n-1},1+2+4+\dots+2^n)$. Define the following sequence of ``Rademacher-like'' vectors:
\begin{itemize}
		\item $r_0 = (1)$;
		\item $r_1 = (\frac{1}{2}, \frac{1}{2})$, $r_2 = (\frac{1}{2}, -\frac{1}{2})$;
		\item $r_3 = (\frac{1}{4}, \frac{1}{4}, \frac{1}{4}, \frac{1}{4})$, $r_4 = (\frac{1}{4}, \frac{1}{4}, -\frac{1}{4}, -\frac{1}{4})$, $r_5 = (\frac{1}{4}, -\frac{1}{4}, \frac{1}{4}, -\frac{1}{4})$;
		\item $r_6 = (\frac{1}{8}, \frac{1}{8},\frac{1}{8},\frac{1}{8},\frac{1}{8},\frac{1}{8},\frac{1}{8},\frac{1}{8},\frac{1}{8})$, $r_7
			 =(\frac{1}{8},\frac{1}{8},\frac{1}{8},\frac{1}{8},-\frac{1}{8},-\frac{1}{8},-\frac{1}{8},-\frac{1}{8})$, \\ $r_8 = (\frac{1}{8},\frac{1}{8},-\frac{1}{8},-\frac{1}{8},\frac{1}{8},\frac{1}{8},-\frac{1}{8},-\frac{1}{8})$, $r_9 =
			 (\frac{1}{8},-\frac{1}{8},\frac{1}{8},-\frac{1}{8},\frac{1}{8},-\frac{1}{8},\frac{1}{8},-\frac{1}{8})$;
	\end{itemize}
	and so on. In general, in the $n$th block we construct $r_i\in\mbb{R}^{2^n}$ for $i\in P_n$ (notice that $2^n=|Q_n|$), they are the first $|P_n|$ many ``Rademacher-like'' vectors in $\mbb{R}^{2^n}$. Define the operation $T\colon \mbb{R}^{2^n}\to L^1[0,1]$ by $T(x)\upharpoonright I_k\equiv 2^nx_k$ for $k<2^n$ where $I_k=\big[\frac{k}{2^n},\frac{k+1}{2^n}\big)$ and $x=(x_k)_{k<2^n}$. Then $T$ is an isometric linear embedding and $\{T(r_i):i\in P_n\}$ is the sequence of the first $n+1$ usual Rademacher functions. In particular the following version of Khintchine's inequality holds true: there is $C>0$ such that for any sequence $(c_i)_{i\in P_n}$ of real numbers
	\[ \bigg\| \sum_{i\in P_n} c_i r_i \bigg\|_1 \leq C \bigg(\sum_{i\in P_n} c_i^2 \bigg)^{1/2} \]
	(this is an immediate corollary of \cite[Theorem 3.25]{Heil}).

Now, we will define a sequence $x=(x_i)_{i\in\om}$ in $\ell_1$ which will represent the {\em Rademacher-ideal} $\mathcal{J}_R=\mc{I}^{\ell_1}_x$. Simply ``shift'' the vectors $(r_i)_{i\in P_n}$ to the interval $Q_n$, put zeros into every other coordinate, and divide it with $n$ (if $n>0$). More precisely, if $i\in P_n$ then let $\mathrm{supp}(x_i)=Q_n$ and
\[ x_i\big(\min(Q_n)+k\big)= \frac{r_i(k)}{n}\;\;\;\text{for}\;\;k<2^n. \]

For $X \sub \omega$ let $A_X = \{n\in \omega\colon P_n \cap X \ne \emptyset \}$.
\begin{thm}$ $
\begin{itemize}
\item[(0)] $\mc{J}_R$ is a tall $F_\sigma$ ideal.
\item[(1)] If $X\in \mc{J}_R$ then $A_X\in \mc{I}_{1/n}$.
\item[(2)] If $A_X\in\mc{I}_{1/\sqrt{n}}$ then $X\in\mc{J}_R$.
\item[(3)] $\mc{J}_R$ is not a summable ideal.
\end{itemize}
\end{thm}
\begin{proof}
(0): Tallness is trivial because $\|x_i\|=1/n$ if $i\in P_n$. Let $\varphi\colon \mc{P}(\om)\to [0,\infty)$ be defined by
\[ \varphi(A)=\sum_{n\in\om}\max\bigg\{\bigg\|\sum_{i\in F}x_i\bigg\|_1\colon \0\ne F\subseteq P_n \cap A \bigg\}.\]
One can see that $\varphi$ is a submeasure and $\mc{J}_R=\mathrm{Fin}(\varphi)$.

\smallskip
(1): Let $\sum_{i\in X}x_i=a$ and assume on the contrary that $\sum_{n\in A_X} 1/n = \infty$. Without loss of generality assume that $X \cap P_n= \{j_n\}$ is a singleton for each $n\in A_X$. Clearly, $|| x_{j_n} ||_1 = 1/n$ for each $n\in A_X$ and thus $\|a\|_1=\big\| \sum_{k \in X} x_k \big\|_1 = \big\| \sum_{n\in A_X} x_{j_n} \big\|_1 = \sum_{n\in A_X} ||x_{j_n}||_1 =  \sum_{n\in A_X} 1/n=\infty$, a contradiction.

\smallskip
(2): Suppose $\sum_{n\in A_X} 1/\sqrt{n} < \infty$. This time we can assume that $X \cap P_n= P_n$ for each $n\in A_X$. We are going to use different (but in $\ell_1$ equivalent, see \cite[Theorem 3.10]{Heil}) definition of unconditional convergence: $\sum (x_n)_n$ is unconditionally convergent if for any choice of signs $(\epsilon_n)_n$ the series $\sum (\epsilon_n x_n)_n$ is convergent, in the classical sense, that is, the sequence of initial subsums is convergent. 

For any choice of signs $(\epsilon_k)_k$, $K_0\in P_{n_0}$, and $K_1\in P_{n_1}$, $K_0< K_1$, $n_0,n_1\in A_X$, using Khintchine's inequality we have
\begin{align*}
\bigg\| \sum_{k\in X\cap [K_0, K_1)} \epsilon_k x_k \bigg\|_1 & =\bigg\|\sum_{i\in P_{n_0}\setminus K_0}\epsilon_ix_i\bigg\|_1+\sum_{n\in A_X\cap (n_0,n_1)} \bigg\| \sum_{i\in P_n} \epsilon_i x_i \bigg\|_1 +\bigg\|\sum_{i\in P_{n_1}\cap K_1}\epsilon_ix_i\bigg\|\\
& \leq C\cdot\!\!\!\sum_{n\in A_X\cap [n_0,n_1]} \bigg(\frac{n+1}{n^2}\bigg)^{1/2}\xrightarrow{n_0,n_1\to\infty} 0  \end{align*}
because of our assumption on $A_X$.

\smallskip
Before (3), we need a general Lemma:
\begin{lem}
Assume that $a=(a_n)_{n\in\om},b=(b_n)_{n\in\om},c=(c_n)_{n\in\om}\in c_0$, $a_n,b_n,c_n\geq 0$ for every $n$, and that $\mc{I}_a\subseteq\mc{I}_b$. Then $\mc{I}_{ac}\subseteq\mc{I}_{bc}$ (where $ac=(a_nc_n)_{n\in\om}$ and $bc=(b_nc_n)_{n\in\om}$).
\end{lem}
\begin{proof}
Fix an $X\in \mc{I}_{ac}$. We can assume that $X=\om$ by restricting our sequences to $X$. We have to show that $\om\in\mc{I}_{bc}$. For every $m$ define $B_m=\{n\colon b_n\geq 2^ma_n\}$, and let $C=\|c\|_0$. 

\smallskip
\textbf{Claim} There is an $m$ such that $B_m\in\mc{I}_b$. 
\medskip

Suppose that $B_m\notin \mc{I}_b$ for every $m$. Then we can find a sequence $(F_m)_{m\in\om}$ of finite sets such that for every $m$ (i) $F_m\subseteq B_m$, (ii) $2^{-m}\leq s_a(F_m)< 2^{-m+1}$, and (iii) $\max(F_m)<\min(F_{m+1})$. Then clearly $B=\bigcup_{m\in\om}B_m\in \mc{I}_a$ but $s_b(F_m)\geq  1$ for every $m$ hence $B\notin \mc{I}_b$, a contradiction.
\medskip

Let $m$ be such that $B_m\in\mc{I}_b$. Then
\begin{align*} \sum_{n\in\om}b_nc_n & \leq C\sum_{n\in B_m}b_n+\sum_{n\notin B_m}b_nc_n\\
& \leq C\sum_{n\in B_m}b_n+ 2^m\sum_{n\notin B_m}a_nc_n\\ &\leq C\sum_{n\in B_m}b_n+2^m\sum_{n\in\om}a_nc_n<\infty.
\end{align*}
\end{proof}

(3): The ideal $\mathcal{I}$ is not a summable ideal. Suppose that $\mc{J}_R=\mc{I}_h$ for some $h\colon \om \to [0,\infty)$. Let $d(n)=s_h(P_n)$ and $e(n) = s_h(P_n)/n$.
According to (2) $\mathcal{I}_{1/\sqrt{n}} \subseteq \mathcal{I}_d$ and by the Lemma (applied for $a(n)=1/\sqrt{n}$, $b(n)=d(n)$, and $c(n)=1/n$) we have $\mathcal{I}_{\frac{1}{n \sqrt{n}}}  \subseteq \mathcal{I}_e$. This just means that $e\in \ell_1$.

For each $n\in \omega$ pick $i_n \in P_n$ of the smallest possible weight with respect to $h$. (1) implies that the set $X = \{i_n \colon n\in\omega\}\notin\mathcal{J}_X$ (simply because $A_X  = \omega \notin
\mathcal{I}_{1/n}$). But $\sum_{n\in X} h(n) \leq \sum_{n\in \omega} \frac{d(n)}{n}= \sum_{n\in\omega} e(n) < \infty$,
a contradiction.
\end{proof}

\begin{que}
Is there a nice characterization of representability in $\ell_1$? Does it imply e.g. that the ideal is $F_\sigma$?
\end{que}

\section{Some related questions}\label{questions}

The topic of this article can be developed in many ways, e.g. by considering characterizations of representability in particular structures. Below we list some problems which we found interesting.
\medskip

\textbf{Representations in $C[0,1]$.} The Banach space $C[0,1]$ of continuous real-valued functions on the unit interval with the sup-norm is (isomorphically) universal for the class of separable Banach spaces, i.e. it contains copies
(via linear homeomorphisms) of all separable Banach spaces. Therefore, all non-pathological analytic P-ideal are representable in $C[0,1]$.

\begin{que}
Is there any ``canonical'' way of representing non-pathological analytic P-ideals in $C[0,1]$? Here ``canonical'' stands for a simple construction of a representation of $\mathrm{Exh}(\sup_{n\in\om}\mu_n)$ in $C[0,1]$ from the defining sequence
$(\mu_n)_{n\in\om}$ of measures on $\om$.
\end{que}

\begin{que}
Assume that all non-pathological analytic P-ideals are representable in a Banach space $X$. Does it imply that $X$ is (isomorphically) universal for the class of all separable Banach spaces?
\end{que}

\textbf{Weak representations.} Another natural way to associate ideals to sequences in topological Abelian groups is the following: If $h:\om\to G$ then let
\[ \widetilde{\mc{I}}^G_h=\Big\{A\subseteq \om\colon \sum h\!\upharpoonright\!A\;\text{is Cauchy}\Big\}.\]

Clearly, $\widetilde{\mc{I}}^G_h$ is always an ideal, and $\mc{I}^G_h\subseteq\widetilde{\mc{I}}^G_h=\mc{I}^{\widetilde{G}}_h$ where $\widetilde{G}\supseteq G$ is the completion of $G$. Of course, if $G$ is complete, then $\mc{I}^G_h=\widetilde{\mc{I}}^G_h$.

For example, let $X$ be a Banach space and consider $X$ with the weak topology $w=\sigma(X,X^*)$. Then we can talk about ideals representable in the completion $\widetilde{X}$ of $(X,w)$. If we are interested in a special case, namely, ideals of the form $\widetilde{\mc{I}}^{X,w}_h$, i.e. ideals of the form $\mc{I}^{\widetilde{X}}_h$ where we work with sequences $h:\om\to X\subseteq\widetilde{X}$ with range in $X$ only, then applying \cite[page 44, Thm. 6]{Diestel} it is easy to see that these ideals are $F_\sigma$.
\begin{que}
Is there any characterization of ideals of the form $\widetilde{\mc{I}}^{X,w}_h$ where $X=c_0,\ell_1,\ell_\infty$?
\end{que}

\textbf{Ideals defined by families of finite sets.} It seems that numerous ideals can be written in a quite special form. For a function $f\colon \om \to [0,\infty)$ and finite set $F\sub \om$ define a measure in the following way: $\mu^f_F (A) =s_f(A\cap F)=\sum_{k\in A\cap F} f(k)$. Now for a family $\mathcal{F} \sub [\om]^{<\om}$ we can define a submeasure $\varphi^f_\mathcal{F} = \sup_{F\in\mathcal{F}} \mu^f_F$.

For instance, consider $2^{<\om}$ instead of $\om$, and let $f,g\colon 2^{\om} \to [0,\infty) $, $f(s)=2^{-|s|}$ and $g(s)=|s|^{-2}$. Then we obtain the following example:
\begin{itemize}
	\item if $\mathcal{F} = [2^{<\om}]^{<\om}$, then $\mathrm{Exh}(\varphi^f_\mathcal{F})$ is the summable ideal;
	\item if $\mathcal{F} = \{$levels$\}$, then  $\mathrm{Exh}(\varphi^f_\mathcal{F})$ is the density ideal;
	\item if $\mathcal{F} = \{$antichains$\}$, then $\mathrm{Exh}(\varphi^f_\mathcal{F})$ is $\mathrm{tr}(\mathcal{N})$;
	\item if $\mathcal{F}$ consists of finite subsets which do not have more than $2^n/n$ elements from $n$-th level, then $\mathrm{Exh}(\varphi^g_\mathcal{F})$ is Farah's ideal.
\end{itemize}

We hope that some families of finite sets (with an appropriate $f$) could give us some interesting ideals.

\begin{que}
Is there any characterization of ideals of the form $\mathrm{Exh}(\varphi^f_\mc{F})$? Or at least in the special cases described above (i.e. with a fixed $f$)?
\end{que}

\textbf{Basic representations.} Let $E$ be the canonical basic sequence of $c_0$. Consider a linear space $X'\subseteq c_0$ equipped with a norm $|| \cdot ||$ and let $X$ be the completion of $(X', ||\cdot||)$. We will say that an ideal $\mathcal{I}$ is
\emph{basically represented} in $X$ if
\[ A\in \mathcal{I} \mbox{ iff } \sum_{n\in A} \alpha_n x_n \mbox{ unconditionally converges in }X \]
for some sequence $(\alpha_n)_{n\in\om}$ from $\mathbb{R}^+$ and a sequence $(x_n)_{n\in\om}$ from $E$ (where we assume that $e = x_n$ only for finitely many $n$'s, for every $e\in E$).

It seems that this representations ties ideals with Banach spaces in a more evident way than the standard representation. E.g.
\begin{itemize}
	\item summable ideals are those basically representable in $\ell_1$;
	\item density ideals are those basically representable in $c_0$;
	\item Tsirelson ideals are basically representable in Tsirelson space(s).
\end{itemize}

Perhaps this approach can be used to construct peculiar Banach spaces. For example we can ask if there is a norm such that $\mathrm{tr}(\mathcal{N})$ is representable in the above way? 
\medskip


\bibliographystyle{alpha}
\bibliography{repr}

\end{document}